\documentclass[a4paper,12pt,twoside]{article}

\usepackage[body={185mm,235mm}]{geometry}
\usepackage{a1}
\usepackage[english]{babel}
\usepackage[latin1]{inputenc} %permite o uso de acentos
\usepackage{amsfonts,amssymb}
\usepackage{amsmath}
\usepackage{graphicx}	
\usepackage{float}
\usepackage{url}
\usepackage{authblk}

\usepackage{makeidx}
\makeindex

\def\mapright#1#2#3{\smash{\mathop{\hbox to
#3{\rightarrowfill}}\limits^{#1}_{#2}}}

\def\mapleft#1#2#3{\smash{\mathop{\hbox to
#3{\leftarrowfill}}\limits^{#1}_{#2}}}

\def\mapright#1#2{\smash{\mathop{\hbox to 0.90cm{\rightarrowfill}}\limits^{#1}_{#2}}}
\def\mapleft#1#2{\smash{\mathop{\hbox to 0.90cm{\leftarrowfill}}\limits^{#1}_{#2}}}

\def\mapleftright#1#2{\smash{\mathop{\hbox to 0.80cm{\leftarrowfill \rightarrowfill}}\limits^{#1}_{#2}}}

\title{Closed, oriented, connected 3-manifolds are subtle equivalence classes of plane graphs
\footnote{AMS classification 05C85 and 05C83 (primary), 57M27 and 57M15 (secondary)}} 

\author[1]{Sóstenes L. Lins}
\author[1]{Diogo B. Henriques}
\affil[1]{Center of Informatics, UFPE}

\date{\today}

\begin{document}

\maketitle

\begin{abstract}
A {\em blink} is a plane graph with an arbitrary bipartition of its edges.
As a consequence of a recent result of Martelli, it is shown that the homeomorphisms classes
of closed oriented 3-manifolds are in 1-1 correspondence with specific classes of blinks. 
In these classes, two blinks are equivalent if they are linked by a finite sequence of 
local moves, where each one
appears in a concrete list of 64 moves: they organize in 8 types,
each being essentially the same move on 8 simply related configurations. 
The size of the list can be substantially
decreased at the cost of loosing symmetry, just by keeping a very simple move type,
the {\em ribbon moves} denoted $\pm \mu_{11}^{\pm}$ (which are in principle redundant). 
The inclusion of $\pm \mu_{11}^{\pm}$  implies that
all the moves corresponding to plane duality (the starred moves), except for $\mu_{20}^\star$
and $\mu_{02}^\star$, are redundant and the coin calculus is reduced to 36 moves on 36 coins.
 A {\em residual fraction link} or a {\em flink}, 
is a new object which generalizes {\em blackboard-framed link}. 
It plays an important role in this work. It is in the aegis of this work to find new important connections 
between 3-manifolds and plane graphs.
%\keywords{Plane graphs \and 3-manifolds \and Blackboard-framed link \and Homeomorphisms}
%\subclass{05C85 and 05C83 (primary) \and 57M27 and 57M15 (secondary)}
\end{abstract}

\section{Introduction}

A {\em blink} is a plane graph with an arbitrary edge bipartition into
two colors (black and gray). {\em Plane} means that it is given embedded in a plane.
Two blinks $B$ and $B'$ are the same if there is an isotopy of the plane onto itself so that 
the image of $B$ is $B'$. The next four blinks are all distinct even tough they have the same 
subjacent graph: 
\raisebox{-2mm}{\includegraphics[width=40mm]{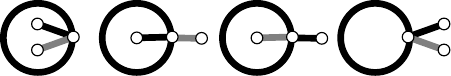}}. Under the
equivalence class generated by the coin moves of Theorem \ref{theo:theorem}
these 4 blinks become the same and each of them
induces the 3-dimensional sphere $\mathbb{S}^3$.

\subsection{Statement of the Theorem}
This paper proves the following theorem:
\begin{theorem}
\label{theo:theorem}
 The classes homeomorphisms of closed oriented connected 3-manifolds are in 1-1 correspondence
 with the equivalence classes of blinks where two blinks are equivalent if one is obtainable from
 the other by a finite sequence of the local moves where each term is 
 one of the 64 moves (not necessary distinct) below\\
 \begin{center}
 \includegraphics[width=16cm]{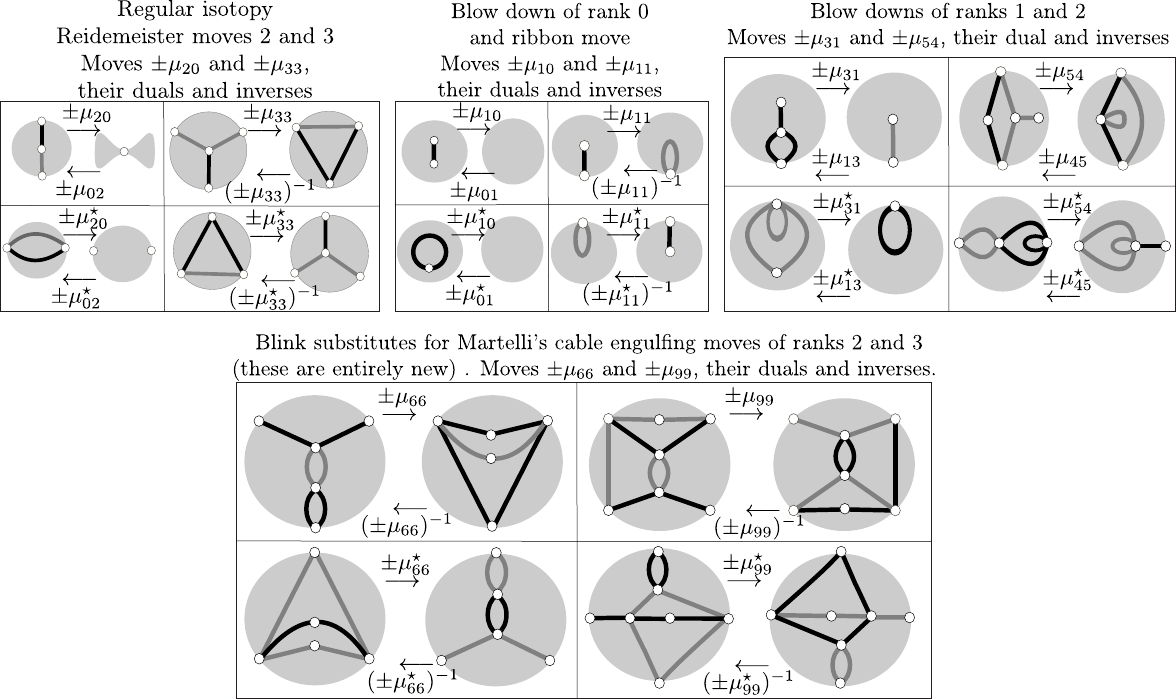} 
  \end{center}
  \end{theorem}

There are 64 local configurations of sub-blinks, each named a {\em coin}, 
divided into 8 families of 8 simply related
coins and also divided into 32 pairs of left-right coins,
A move replaces the left (right) coin of a pair by its right (left) 
coin. Thus, the number of moves is equal to the number of coins.
Boundary and internal vertices of the coins are shown as small white disks. 
The complementary sub-blink in the exterior a coin is completely arbitrary; 
its intersection with the corresponding internal coin is a subset of the set of 
attachment vertices in the boundary of the coin.
The support of the coins
are disks, except in the the right coin  
of $\mu_{20}$, case in which is a pinched disk. 
The number $k$ of the attachment vertices  satisfies $k\in \{0,1,2,3,4\}$. 
 
 \subsection{Organization of the paper}
 In Section \ref{sec:motivationtopprelim} the motivation, the topological 
 preliminaries and an epistemological view of the work are discussed.
 The proof of the Theorem \ref{theo:theorem} 
 is given in Section \ref{sec:prooftheorem}.
 A reduced but sufficient form of this coin 
 calculus having 36 coins (and moves) is obtained in Section \ref{sec:roleribbonmoves}.
 In Section \ref{sec:conclusion}a complete census with no duplicates of
 the closed, oriented, connected, prime 3-manifold induced by blinks up to 8 edges.

\section{Motivation and topological preliminaries}
\label{sec:motivationtopprelim}
\subsection{Motivation}

In his Appendix to part 0, J. H. Conway in his famous book {\em On Numbers and Games}, 
\cite{onnumbersandgames}, says: {\em ``This appendix is in fact a cry for a Mathematician
Liberation Movement!

Among the permissible kinds of constructions we should have:\\
(i) Objects may be created from earlier objects in any reasonable constructive fashion.\\
(ii) Equality among the created objects can be any desirable equivalence relation.''
} 
 
This paper is in the confluency of two significant research fields:
the topological study of closed orientable 3-manifolds and the combinatorial study of plane 
graphs.  The result proved here provides a glimpse, in the spirit of 
Conway's quotation, to effectively
enumerate once each closed orientable 3-manifolds. 
Blinks are easy to construct from simpler blinks, and their 
isomorphism problem can be solved by a polynomial
algorithm which finds, via a few fixed conventions (lexicography), a numerical {\em code} for it. 
What can be more desirable than an equivalence relation on such simple mathematical 
objects that captures the subtle and difficult computational topological notion of
factorizing any homeomorphism between two closed, oriented, connected 3-manifolds?
 
\subsection{Topological preliminaries}

This subsection contains the basic topological material needed. 
It is primarely intended to the 
combinatorially oriented readers, unfamiliar with the fundamental definitions of knots, 
links and framed links.
Some unfamiliar notation  is also introduced and 
definitions which are new (such as the flinks) and that will be 
used throughout the paper. Moreover,a short historic overview of the known results is presented.

\subsubsection{Knots and links into $\mathbb{R}^3$ and into $\mathbb{S}^3$}

A {\em knot $K$} is an embedding of a circle,  $\mathbb{S}^1$, into $\mathbb{R}^3$
(or into $\mathbb{S}^3$, the boundary of a 4-dimensional ball).
The {\em unknot} is a knot which is the boundary of a disk.
A {\em link with $k$ components} is an embedding of a disjoint union of $k$ 
copies of $\mathbb{S}^1$,
$\left(\cup_{i=1}^k \mathbb{S}^1_i \right)$,
into $\mathbb{R}^3$ ((or $\mathbb{S}^3$) with disjoint images. 
In this way, a knot is a particular case of a link: one which
has one component. Fig. \ref{fig:R3S3} shows a 1-1 correspondence
$K \leftrightarrow K'$ between knots into $\mathbb{R}^3$ and into $\mathbb{S}^3$. 
By abuse of language, a knot is identified with its image.

\begin{figure}[H]
\begin{center}
\includegraphics[width=9.0cm]{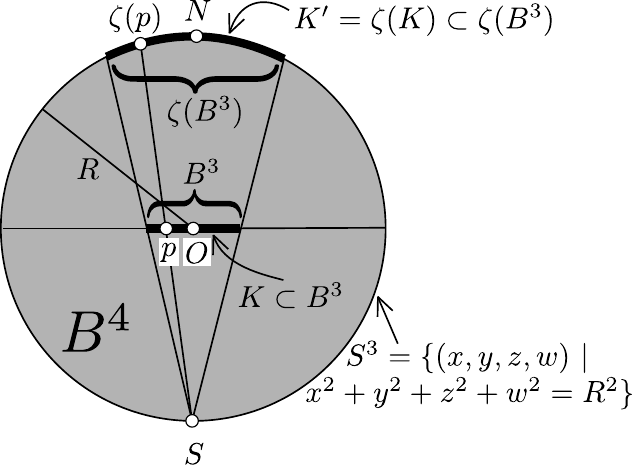} 
\caption{\sf How a knot $K$ into $\mathbb{R}^3$ 
is related to a knot $K'$ into $\mathbb{S}^3$: 
the 3-sphere $\mathbb{S}^3$
is the boundary of a 4-ball $\mathbb{B}^4$; the knot $K$ is contained 
in a ball $\mathbb{B}^ 3$ of radius $r$ centered in the origin $O$ and 
contained in the equator $\{(x,y,z,w) \in \mathbb{B}^4 \ | \ w=0\}$; by making 
$\frac{R}{r}$ big knots $K$ (which the south pole of the 4-ball)
and $K'$ are as close as being isometric as desired. $K'$ is the image of $K$ under the
stereographic projection $\zeta$ centered at the south pole $S$. In this work, knots into 
$\mathbb{R}^3$ and into $\mathbb{S}^3$ are worked with. Thus the 
easy correspondence between the two types is welcome.
}
\label{fig:R3S3}
\end{center}
\end{figure}

Links can be presented with profit by their {\em decorated 
general position projections} into the $xy$-plane $\mathbb{R}^2$,
by simply making 0 the $z$-coordinate. As in Fig. \ref{fig:R3S3}, the knot is inside 
$\mathbb{B}^3 \subset \mathbb{R}^3 $,
having made the 4th coordinate $w$ equal to 0.
Here {\em general position} means that 
in the image of the link there is no triple points and that at each neighborhood of
a double point is the transversal crossing of two segments of the link, named {\em strands}. 
{\em Decorated} means that we keep the information
of which strand is the upper one, usually by removing a piece of the lower strand. In this paper
yet another way to decorate the link projections is used: the images of the link components 
are thick black curves
and the upper strands are indicated by a thinner white segment inside the thick 
black curve at the crossing (see left side of Fig. \ref{fig:knotandframedknot}).

\subsubsection{Framed knots, ribbons, framed links and blackboard-framed links into $\mathbb{R}^3$}
A {\em framed knot} is an embedding of  
$\mathbb{S}^1 \times [-\epsilon,+\epsilon]$  
into $\mathbb{R}^3$ (or $\mathbb{S}^3$), for an arbitrarily fixed $\epsilon>0$. 
A framed knot is also called {\em a ribbon}.
The {\em base knot} of a ribbon
is the ribbon restricted to $\mathbb{S}^1 \times\{0\}$. 
A {\em framed link} is a collection of ribbons with disjoint images. The ribbons 
used are prepared by isotopies so that their projections remain with constant
width $2\epsilon$. Fig. \ref{fig:GettingImmersedBandSameWidth} 
shows how to achieve this condition.

\begin{figure}[H]
\begin{center}
\includegraphics[width=6.0cm]{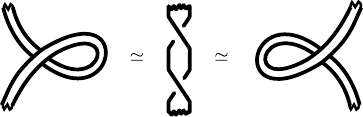} 
\caption{\sf Getting a constant width immersion of a ribbon projection into $\mathbb{R}^2$. 
Each 360-degree rotation of the ribbon in the space is ambient isotopic to
a curl in the ribbon (in two different ways) so as to maintain  
constant the width immersion of the ribbon projection.
In particular, after making these isotopies, the intersection of the images of
$\mathbb{S}^1 \times\{-\epsilon\}$
and $\mathbb{S}^1 \times\{+\epsilon\}$
have 4 distinct points (it used to have 2 points) 
near each crossing of the base link and there are no other 
crossings in the immersed ribbons. 
}
\label{fig:GettingImmersedBandSameWidth}
\end{center}
\end{figure} 

Note that a blackboard-framed link is not a framed link.
Indeed, the blackboard-framed is the base link of its framed link, see Fig. 
\ref{fig:knotandframedknot}.

\begin{figure}[H]
\begin{center}
\includegraphics[width=8.5cm]{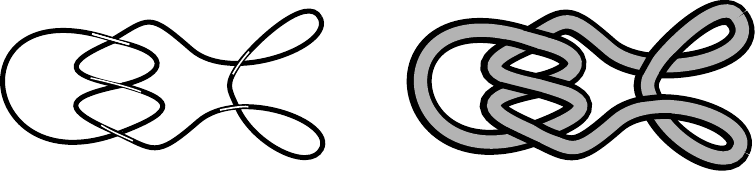} 
\caption{\sf A knot and its corresponding framed knot (its ribbon) in an adequate projection,
one which maintains the width of the ribbon.
Taking the base knot of this ribbon produces back 
the {\sf blackboard-framed link}, seeing on the left.
}
\label{fig:knotandframedknot}
\end{center}
\end{figure}

The isotopy class of a framed link is determined by assigning an 
integer to each component of the link which is equal to the {\em linking number} of the
two components of the boundary of each of its ribbon oriented in the same way. The linking number 
can be obtained from an arbitrary decorated general position projection of the oriented band:
it is equal to half of the algebraic sum of the $(\pm)$-signs of the 
crossings of distinct boundary components. The convention is that
\raisebox{-3mm}{\includegraphics[width=3cm]{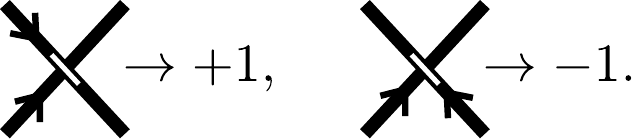}}
The linking number is an invariant, that is, it does not depend on 
the particular projection used. This is proved by K. Reidemeister in its 1932 book
on Knot Theory, \cite{reid1932}. He isolates three local moves $r_1, r_2, r_3$ that
are enough to finitely factor any arbitrary ambient isotopy between two decorated general 
position projections of the same link:
\raisebox{-3mm}{\includegraphics[width=7cm]{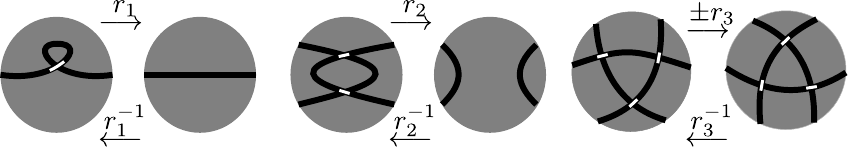}}. The {\em self-writhe 
of a  projected link component} is the algebraic sum of the signs of the self-crossings
of that component. 

As a matter of fact, Theorem \ref{theo:theorem} which 
is proved in this work is the counterpart for 3-manifolds
of Reidemeister Theorem. Note that in the coin calculus depicted together
with flinks in Fig. \ref{fig:flinkandlinktogether}
there are 8 versions of each of Reidemeister moves $r_2$ and $r_3$. Note also that 
Reidemeister move $r_1$ is not used at all. 
The equivalence class of decorated general position link projection generated
by $r_2^{\pm 1}$ and $r_3^{\pm 1}$ is called {\em regular isotopy}. It plays an important role
in the computation of the Jones invariants via Kauffman's bracket \cite{kauffman1987state}.

A {\em blackboard-framed link} is an {\em adequate projection} 
of a framed link prepared by isotopy so that 
the base link is projected as a decorated general position one and its ribbons are immersed
with constant width $2\epsilon$, see Fig. \ref{fig:GettingImmersedBandSameWidth}.  
Moreover, we adjust the number of curls and their signs 
(in the base link) so that the self-writhe of each component coincides with the linking number of 
the two boundary components of its ribbon. The advantage of the blackboard-framed link is that
we no longer have to worry about assigning numbers to the components. These integer 
numbers are induced by the plane of projection. The importance of this concept was 
advocated by L. Kauffman in a number of works, including \cite{kauffman1991knots} and 
\cite{kauffman1994tlr}. Its generalization follows.

\subsubsection{Flinks and its relation with blinks}
A new object, defined in this work is a {\em flink}. It generalizes the notion of
blackboard-framed link in an adequate way (an invariance under 
Kirby's handle sliding move) to be made clear, see Fig.\ref{fig:handleslide}. 
Flink is a dutch word meaning
{\em significantly}. It is also an acronym for residual {\bf f}raction {\bf link}.
A {\em residual fraction} is either an irreducible fraction $\frac{p}{q}$ with $0<p<q$, else
$\pm 1/0 =\pm \infty$, else $0/1=0$. 
The {\em ribbon move}
in a decorated general position link projection in $\mathbb{R}^2$ is the move defined by
the following change in a pair of coins:
\raisebox{-4mm}{\includegraphics[width=2.5cm]{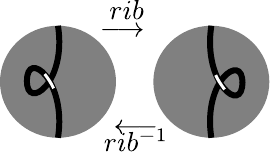}}.
A {\em flink} is the equivalence class of fraction decorated general position link
projections under regular isotopies and ribbon moves. {\em Fraction decorated} means 
decorated and having a residual fraction assigned to each component. A {\em $0$-flink}
is one so that each component has 0 as its residual fraction. Note that $0$-flink
and blackboard-framed link are the same concept.
Fig. \ref{fig:linkblink} shows that there is a 1-1 correspondence between 0-flinks and blinks.

The {\em surgery coefficients $p'$ and $q$ associated to a component of a flink} satisfy
(by definition) $\frac{p'}{q} = w + \frac{p}{q}$ where $w$ is the self-writhe of the component
and $\frac{p}{q}$ is its residual fraction.

\begin{figure}[H]
\begin{center}
\includegraphics[width=15cm]{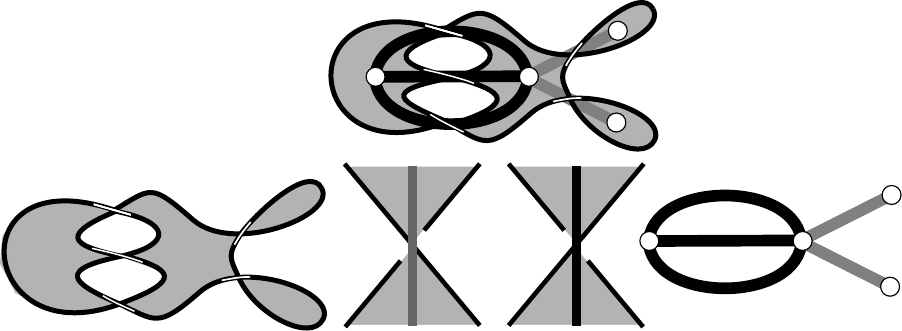} 
\caption{\sf From 0-flink to blink and back: 
the projection of any link can be 2-face colorable into white and gray with the infinite
face being white so that each subcurve between two crossing have their incident 
faces receiving distinct
colors. The above figure shows how to transform a link projection into a blink (with thicker edges
than the curves representing the link projection). 
The vertices of the blink are distinguished fixed
points represented by white disks in the interior of the gray faces. 
Each crossing of the link projection becomes an edge in the corresponding blink, 
An edge of the blink is gray if the upper strand 
that crosses it is from northwest to southeast, it is black if the the upper strand 
that crosses it is from northeast
to southwest. The inverse procedure
is clearly defined. In fact, the link is the so called {\em medial map} of the blink.
Thus we have a 1-1 correspondence between 0-flinks
and blinks. The (complete) blink at the right induces Poincaré's sphere, 
the spherical dodecahedron space. The expressibility of a blink is powerful: quite
complicated 3-manifolds are induced by simple blinks.}
\label{fig:linkblink}
\end{center}
\end{figure} 

\subsubsection{Lickorish's groundbreaking result}
In a grounding breaking work (1962),  W.B.R. Lickorish, 
\cite{lickorish1962representation}, proved that
any closed, oriented, connected $\mathbb{M}^3$ has inside it a finite number
$k$ of disjoint solid tori each in the form of a homeomorphic image 
of $\mathbb{S}^ 1  \times[-\epsilon,+\epsilon]\times[-\delta,+\delta]$,
denoted by  $(\mathbb{S}^ 1  \times[-\epsilon,+\epsilon]\times[-\delta,+\delta])_i$
so that, where $\mathbb{S}^ 3$ is the 3-dimensional sphere, 
$$\mathbb{M}^ 3\backslash \bigcup_{i=1}^k 
(\mathbb{S}^ 1  \times[-\epsilon,+\epsilon]\times[-\delta,+\delta])_i=
\mathbb{S}^ 3\backslash \bigcup_{i=1}^k (\mathbb{S}^1 
\times[-\epsilon,+\epsilon]\times[-\delta,+\delta])_i.$$
As a consequence, each closed oriented 3-manifold can be obtained from $\mathbb{S}^3$ 
by removing a subset of disjoint solid tori and pasting them back in a different way.
The most general parameters to identify the pasting is a pair of integers for each 
component of the link, named {\em surgery coefficients}. 
A filling algorithm algorithm applies component by component, is named 
{\em $(\pm p,q)$-Dehn filling}, and is explained in Fig. 
\ref{fig:surgerycoefficientsfromflink}. This Figure explains how to get
the surgery coefficients from the flink.

\begin{figure}[H]
\begin{center}
\includegraphics[width=14.0cm]{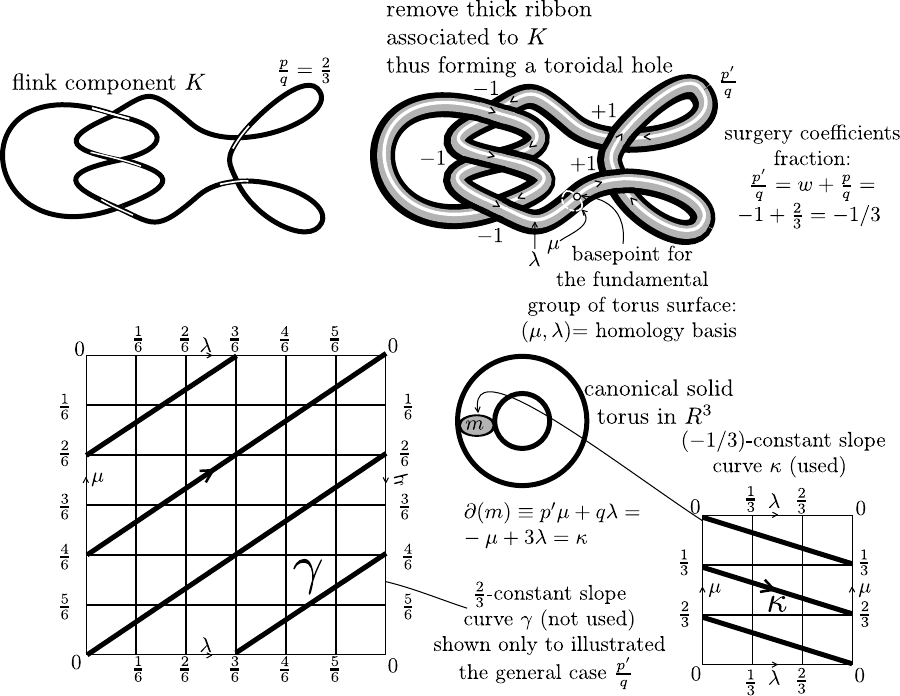} 
\caption{\sf How to obtain a closed oriented connected 
3-manifold from a flink in $\mathbb{S}^3$:
remove the thick ribbons correponding to all components of the flink.
Repeat the following $(p',q)$-Dehn filling for each component.
Let $\lambda = K \times \{0\} \times\{+\delta\}$ and let $\mu$ be a closed curve that 
is contractible in the thick ribbon but not in its boundary. The pair of closed curves
$(\mu, \lambda)$ form a basis for the fundamental group of the boundary of the thick ribbon.
Orient $\mu$ arbitrarily and $\lambda$ so that $K$ equally oriented has linking number 1 with $\mu$.
Note that $\lambda$ remains parallel to $K$ and never touches the boundary of the ribbon.
Let $\kappa$ be the constant slope closed curve homotopic to $p\mu + q\lambda$. Consider 
a canonical solid torus embedded into $\mathbb{R}^3$ and let $m$ be its meridian.
Consider a homeomorphism that identify curves $\kappa$ and the boundary of $m$, $\partial(m)$.
This homeomorphism is univocally extensible to identify the boundaries of the thick
ribbon and of the canonical solid torus, thus closing the toroidal hole: indeed, after
identifying $m \times [-\zeta,+\zeta]$ (for a small $\zeta >0$) in the solid torus and in the toroidal 
hole in the 3-manifold what remains to identify are two 3-balls. This has a unique 
solution up to isotopy.
}
\label{fig:surgerycoefficientsfromflink}
\end{center}
\end{figure} 

For proving the theorem, and this is important in the present work, the fillings used 
by Lickorish satisfy $q=1$. He call these surgeries {\em honest}.
Actually, Lickorish's result had been proved 2 years before 
by A. H. Wallace \cite{wallace1960modifications} by using differential geometry. However
it was the purely topological flavor of Lickorish's proof that spurred
the subsequent developments. Also, Lickorish does not state his theorem in this way.
The form used for the solid tori is convenient because of its simple relation with flinks.
It is inspired in the lucid account by J. Stillwell of the Lickorish's theorem 
given in \cite{stillwell1993classical}. Each each solid torus to be removed
from $\mathbb{S}^ 3$ is a thich ribbon 
$K \times[-\epsilon,+\epsilon] \times[-\delta,+\delta]$, where 
$K$ is (the image of) a component of the flink, with fixed small positive
constants $\epsilon>\delta>0$. Each section $k \in K$, 
$\{k\} \times[-\epsilon,+\epsilon] \times[-\delta,+\delta]$, is a rectangle whose
$\delta$-sides are parallel to the $z$-axis. Thus the projections of the thick 
ribbon and the of the ribbon coincide.

\subsubsection{Kirby's famous calculus of framed links}
In 1978 R. Kirby published his, to become famous, calculus of framed links, \cite{kirby1978calculus}.
The gist of this paper is that two types of moves 
are enough to go from any framed
link inducing a closed oriented 3-manifold to any other such link inducing the same manifold.
One of the moves is absolutelly local: creating or cancellating an arbitrary new unknotted 
component with frame $\pm 1$ separated from the rest of the link by an $\mathbb{S}²$.
The other type of move, {\em the band move}, \cite{kirby1978calculus} \cite{kauffman1991knots}) 
{\em (or handle sliding)} is non-local and infinite in number.

\begin{figure}[H]
\begin{center}
\includegraphics[width=14cm]{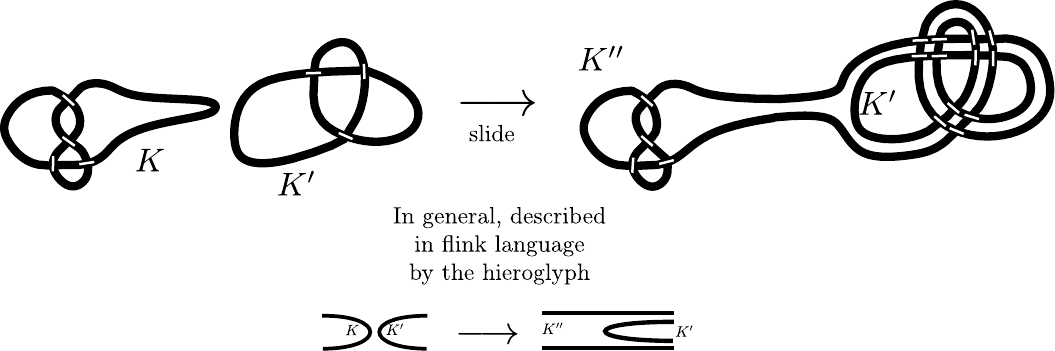} 
\caption{\sf Kirby's {\em handle slide move}, also known as the {\em band move}. 
In this work the handle slide move is applied only to flinks, so there is no issue
about surgery coefficients.
Given $K$ and $K'$ two distinct components of a
flink start by making a close parallel copy of $K'$ in such a way that 
it forms an immersed band with its originator $K$.
Let $K''$ be the connected sum of $K$ and the copy of $K'$. 
The connected sum is defined by a em new band which is a
thin rectangle arbitrarily embedded into $\mathbb{S}^3$, so as to miss the link. 
The short sides of this band 
are attached to $K$ and to the copy of $K'$ and then, recoupled in the other way.
The new band can be quite complicated because it may wander
arbitrarily (as long as it misses the link) 
in $\mathbb{R}^3$ in its way to connecting
the two components. More details in Kauffman's book, \cite{kauffman1991knots}. 
The property of flinks that made me introduce the concept is that the
{\bf residual fractions of the flink are invariant under the Kirby's band move.}
This move then can be depicted in all its generality, 
via the hieroglyph shown in the bottom part of the 
Figure. See Section 12.3 of \cite{kauffman1994tlr}. However,
this hieroglyphic move is non-local since the exterior of the hieroglyph changes, and 
beacause there are infinite exteriors, there are infinite Kirby's band moves. Fenn-Rourke
reformulation, treated next, provide an infinite sequence of trully local moves.
}
\label{fig:handleslide}
\end{center}
\end{figure} 

\subsubsection{Fenn-Rourke reformulation of Kirby's calculus}

\begin{figure}[H]
\begin{center}
\includegraphics[width=16.5cm]{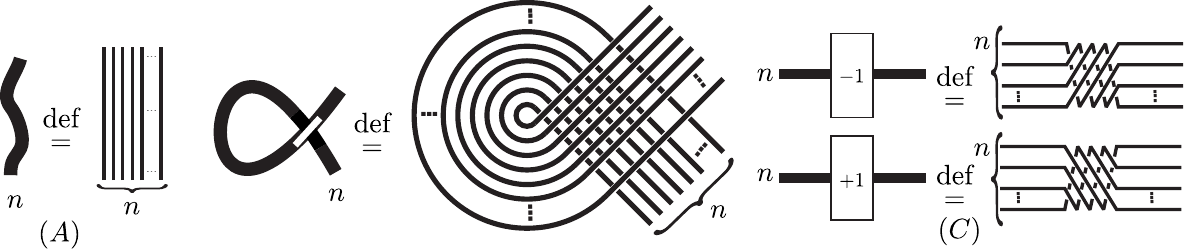} 
\caption{\sf Notation for special disk neighborhoods of general position decorated link projections}
\label{fig:somenotation}
\end{center}
\end{figure} 

In 1979 R. Fenn and C. Rourke (\cite{fenn1979kirby}) show that Kirby's moves could 
be replaced by an infinite sequence of a single type of move (a {\em blow down move})
indexed by $n$, which is depicted at the left side of Figure \ref{fig:FennRourkeAndKauffmanMove}. 
In a blow down move the number of components decreases by 1.
This has been a very useful reformulation with many applications, 
including Martelli's calculus (soon to be treated)
which uses it instead of the direct moves of Kirby. 

\begin{figure}[H]
\begin{center}
\includegraphics[width=11.5cm]{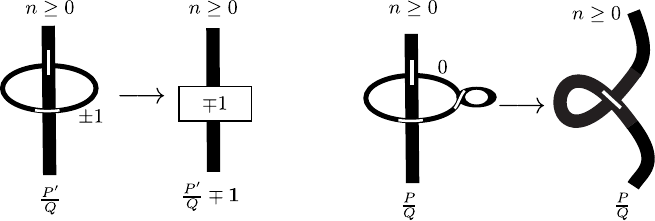} 
\caption{\sf Fenn-Rourke infinite sequence of blown-down moves on fraction decorated
links into $\mathbb{R}^3$ and 
their counterpart in terms of flinks, which generalizes 
Kauffman's blackboard-framed 
links. By definition,  $\frac{P'}{Q} = (\frac{p'_1}{q_1}, \ldots \frac{p'_n}{q_n})$
and $\frac{P}{Q} = (\frac{p_1}{q_1}, \ldots \frac{p_n}{q_n})$, where $\frac{p_i'}{q_i}$
are generalized fractions associated with the $i$-th component of the link
in $\mathbb{R}^3$ and 
$\frac{p_i}{q_i}$ is the associated residual fraction 
of the link projection into $\mathbb{R}^2$. 
The blown-downs constitute an
infinite sequence of local moves. The special case $n=0$ of these moves is also
defined and replace an isolated $\pm 1$-framed component of the unknot with one crossing by nothing.
Martelli replaced the infinite sequence by the first three and two new simple
moves $A_3$ and $A_4$, described in Fig. \ref{fig:Martelli}.}
\label{fig:FennRourkeAndKauffmanMove}
\end{center}
\end{figure} 

\subsubsection{Kauffman's idea to let the plane induce the integer framing}
In the beginning of the 1990's L. Kauffman
presented (\cite{kauffman1991knots}) a completely planar diagramatic way to deal with the 
calculus of Kirby and its reformulation by Fenn and Rourke.
The basic idea comes from the fact that every 3-manifold is
induced by surgery on a framed link which has {\em only finite integer framings}. 
This characterize the {\em handle surgeries}. According to 
Rolfsen Lickorish call each of these a
{\em honest surgery}, page 262 of \cite{rolfsen2003knots}.
The proof that we can get any manifold by surgery on integer framed links 
uses, as a lemma,  the fact that it is possible to
modify the framed link maintaining the induced 3-manifold so that every component becomes unknotted. 
A proof of this lemma appears \ref{fig:switchingcrossing}. If a component is unknotted then 
it is simple to modify the link so that each component gets an integer framing, without 
distub the integrality of the framing of other components. See Theorem 
\ref{theo:twisthomeo}. So, without loss of
generality we may suppose that all the components have finite integers as framings. Kauffman's proceeds
by adjusting each component by attaching to it a judicious number of curls so that the required framing of
a component coincides with the algebraic sum of its self-crossings. 
By specifying that the link is {\em blackboard-framed}, we no longer need the integers to
specifies the framing. They are a consequence.
In this work only blackboard-framed projections or their generalization, flink, are used.
Flink is a convenient generalization of blackboard-framed link because
their residual fractions remain constant in Kirby's calculus, in Fenn-Rourke reformulation,
and in Martelli's calculus, treated next.

\subsubsection{Martelli's finite calculus on fractionary framed link}
In an important recent paper B. Martelli \cite{martelli2012finite} 
presented a local and finite reformulation of
the Fenn-Rourke version (\cite{fenn1979kirby}) of Kirby's calculus \cite{kirby1978calculus}. 
This calculus is presented in Fig. \ref{fig:Martelli}. It remains
to be seen the consequences of Martelli' s result for obtaining new 3-manifold invariants.
A possible door for obtaining such invariants are generalizations of the combinatorial 
approach to get WRT-invariants, justified in \cite{kauffman1994tlr} and extensely
used in \cite{lins1995gca} and in \cite{lins2007blink}. To find such a generalization 
one has to take advantage of the specific sufficient local Martelli's calculus now available 
(or the coin calculus on blinks). The WRT-invariants are obtained by  
hiding in the Temperley-Lieb algebra the infinite cases of Kirby's band move, 
as pioneered by Lickorish in \cite{lickorish1991three}. 
See also page 144 of the join monography of L. Kauffman and myself, \cite{kauffman1994tlr}.
Finding such generalization of the WRT-invariants still seems to be a formidable task.
However, Martelli's theorem and the coin calculus on blinks makes it conceivable.

\begin{figure}[H]
\begin{center}
\includegraphics[width=16.5cm]{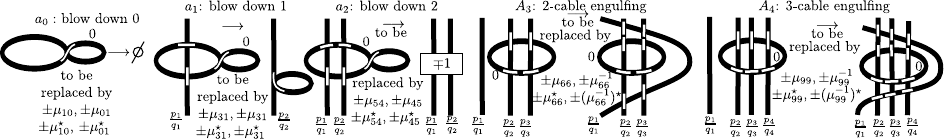}
\caption{\sf 
Martelli's calculus on flink language:
Note that for the internal
components of all the above moves the residual fraction are 0.
Martelli' proves that by keeping only
the blown-down of ranks 0, 1 and 2 and replacing all the remaining infinite sequence by two
new moves,  called 2- and 3-cable engulfing (denoted by $A_3$ and $A_4$), a sufficient 
calculus for factorizing 
homeomorphisms between closed oriented and connected 3-manifolds is achieved solely 
in terms of the above 5 local moves, their inverses, the regular isotopies moves and the ribbon moves. 
Moves $A_3$ and $A_4$ do not translate into blink moves
because their left sides are disconnected. What makes this work possible is the replacement of
these non-connected configurations by equivalent moves $a_3$ and $a_4$ 
so that blink translations become available. 
The equivalences $a_3\equiv A_3$ and $a_4\equiv A_4$ are proved in Fig. \ref{fig:proofequivalencea3A3a4A4} 
where the diagrams for the moves
appear right angle rotated relative to this figure. Observe that, in the flink language,
all the residual fractions of Martelli's calculus remain invariant.
}
\label{fig:Martelli}
\end{center}
\end{figure}

\subsubsection{Unknotting a component}
Any knot given by an decorated general position projection can be unknotted by a subset of 
crossing switches: starting in a non-crossing and going along the knot
make sure that the first passage through a crossing becomes an upper strand
(by switching the crossing if necessary). The result is clearly a projection of an unknot.
Let $F$ be a flink and $F_i$ be one of its component which is knotted. It is possible 
to modify $F$ at the cost of introducing new unknotted components so that $F_i$ becomes
unknotted and the induced 3-manifold does not change. This is a consequence of the 
above algorithm to unknot any knot and of Kirby's calculus on flinks, as shown in Fig.
\ref{fig:switchingcrossing}, which is adapted to fit the flink language, from \cite{kauffman1994tlr}.
\begin{figure}[H]
\begin{center}
\includegraphics[width=16.5cm]{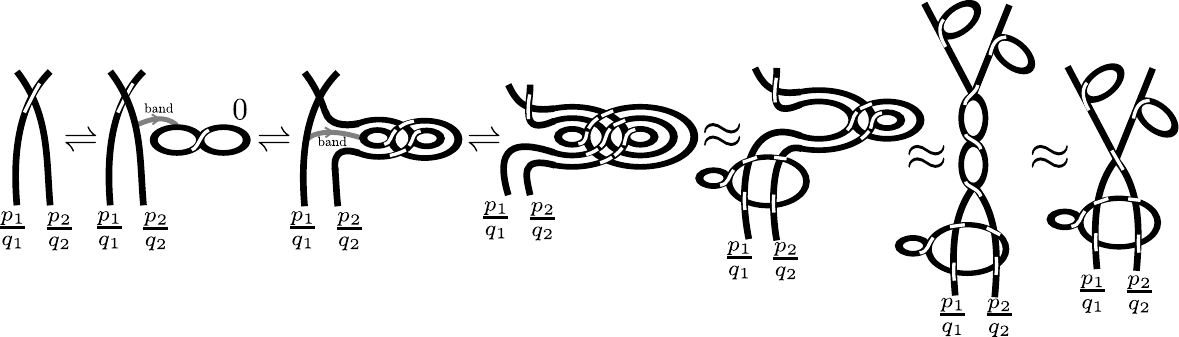}
\caption{\sf Switching a crossing at the cost of adding an unknoted component and
some curls: this is achieved by applying Kirby's move of type 1 
and twice Kirby's moves of type 2 as well
as some isotopies. Note that the residual fractions stay invariant.
}
\label{fig:switchingcrossing}
\end{center}
\end{figure} 

\subsubsection{Obtaining a 0-flink inducing the same manifold as any input flink}
Let $F = F_1 \cup F_2 \cup \ldots \cup F_k$ be a flink and let $r'_j$ be the 
fraction of the surgery coefficients given, as already defined, by the sum of 
the residual fraction $r_j$ and the self-writhe of the $j$-th component $F_j$. 
Let $F_i$ be an unknotted component.
\begin{lemma}
The two following operations on a flink $F$ maintain the induced 3-manifold:
\begin{itemize}
\item Create or cancell a component with residual fraction $\pm \infty$.
\item Effect a $t$-full-twist about $F_i$ changing
judiciously the link and the surgery coefficient fractions $r'_k$'s.
\subitem The surgery coefficient fractions change as follows:
\subsubitem Component $F_i$ of the twist: $r_i''=1/(t+(1/r'_i))$
\subsubitem Other components $F_j$: $r_j''=r'_j + t(\ell kn(F_i,F_j))^2$. 
\subitem The link changes as follows:
\subsubitem Effect $t$ positive or negative full twists, according 
to the $\mp$-sign of $t$, in the cable of parallel lines
encircled by $F_i$.
\end{itemize}
\end{lemma}
\begin{proof}
The proof of this result is given in \cite{rolfsen2003knots}.
\end{proof}

\begin{theorem}
\label{theo:twisthomeo}
Given any flink, there exsits a 0-flink inducing the same 3-manifold
obtainable by a polynomial algorithm in the product of its 
residual fraction denominators.
\end{theorem}

\begin{proof}
Let $F_i$ be an knotted component of a flink $F$ which have a non-null 
residual fraction  $\frac{p}{q}$. Switch crossings as in 
ref{fig:switchingcrossing} so that $F_i$ becomes unknotted,
with the same residual fraction $\frac{p}{q}$ and respective surgery coefficient fraction
$\frac{p'}{q}$. If $p'=\pm 1$ then use Lemma \ref{theo:twisthomeo} with $t=\mp q$. The 
surgery coefficients fraction of $F_i$ becomes $r_i''=1/(t+(1/r'_i))=1/(\mp q+\pm q)=1/0=\infty$. 
And component $F_i$ can be removed. Let $|p'|>1$. We have 
$r_i''=1/(t+(1/r'_i))=1/(t+(q/p'))$. Define $t$ to be the integer so that $0<t+q/p'=q'/p'<1$.
The new surgery coefficient fraction of $F_i$ is $r_i''=p'/q'$, with $q' < q$.
Repeating the procedure a number of times bounded by $q$ we arrive at $q'=1$ or $q'=0$.
Observe that the surgery coefficient fractions of the new components are integers and 
the one of the components $F_j$ linked with $F_i$ change by an integer. 
Thus we get to a flink whose all surgery coefficients fractions have $q=1$.
Therefore, by curl adjusting, a 0-flink inducing the same 3-manifold 
is obtained.
\end{proof}

\section{Proof of the Theorem}
\label{sec:prooftheorem}

\begin{figure}[H]
\begin{center}
\includegraphics[width=14cm]{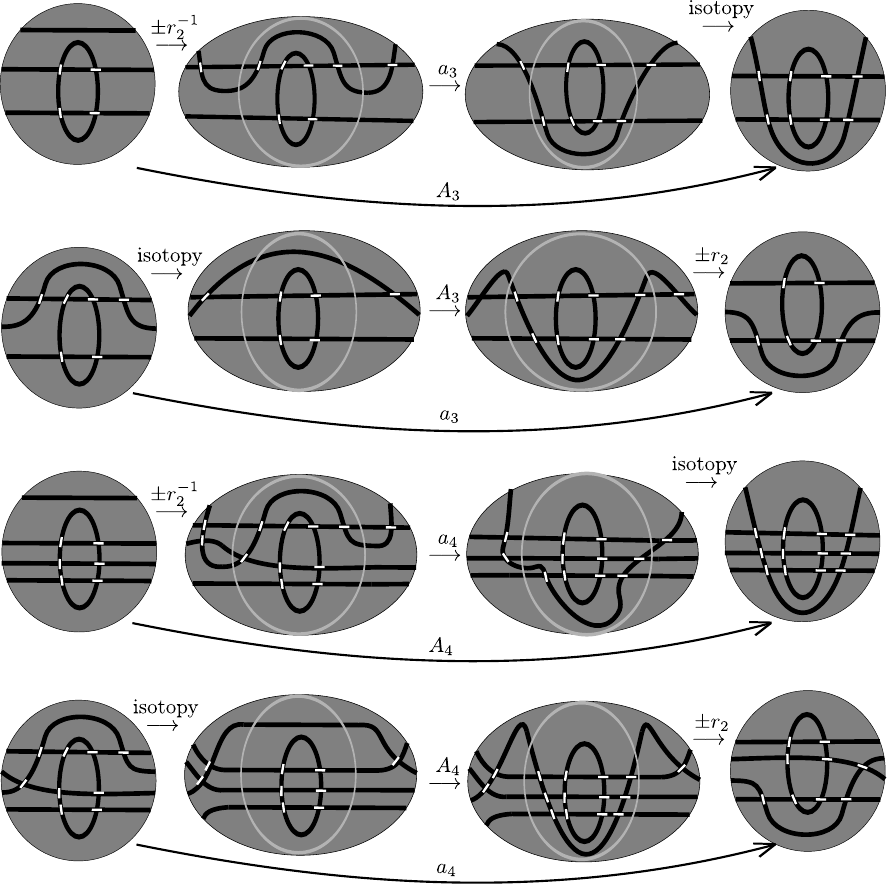} 
\caption{\sf A proof that in the presence of 
$\pm r_{2}^{\pm 1}$, the equivalences   
$a_3 \equiv A_3$ and $a_4 \equiv A_4$ hold}
\label{fig:proofequivalencea3A3a4A4}
\end{center}
\end{figure} 

\begin{lemma}
 In the presence of Reidemeister moves 2, generaly denoted by $\pm r_2^{\pm -1}$,
 moves $\pm a_3$ and $\pm A_3$ are equivalent and so are
 moves $\pm a_4$ and $\pm A_4$.
\end{lemma}
\begin{proof}
 We refer to  Fig. \ref{fig:proofequivalencea3A3a4A4}. Its first line proves that $\pm a_3 \Rightarrow \pm A_3$.
 The second line proves that $\pm A_3 \Rightarrow \pm a_3$. The third line proves 
 that $\pm a_4 \Rightarrow \pm A_a$. The last line proves that $\pm A_4 \Rightarrow \pm a_4$. 
\end{proof}

\begin{proof} ({\bf of Theorem \ref{theo:theorem}})\\
In Fig. \ref{fig:flinkandlinktogether} we draw all the moves 
for the revised Martelli's moves on pairs of distinctly 2-colored 0-flinks 
and the respective blinks superimposed. The result follows by removing the 
0-flink moves leaving only the blink moves which are redrawn up to isotopy,
in the lower part of the figure.
\end{proof}

\begin{figure}[H]
\begin{center}
\includegraphics[scale=0.7]{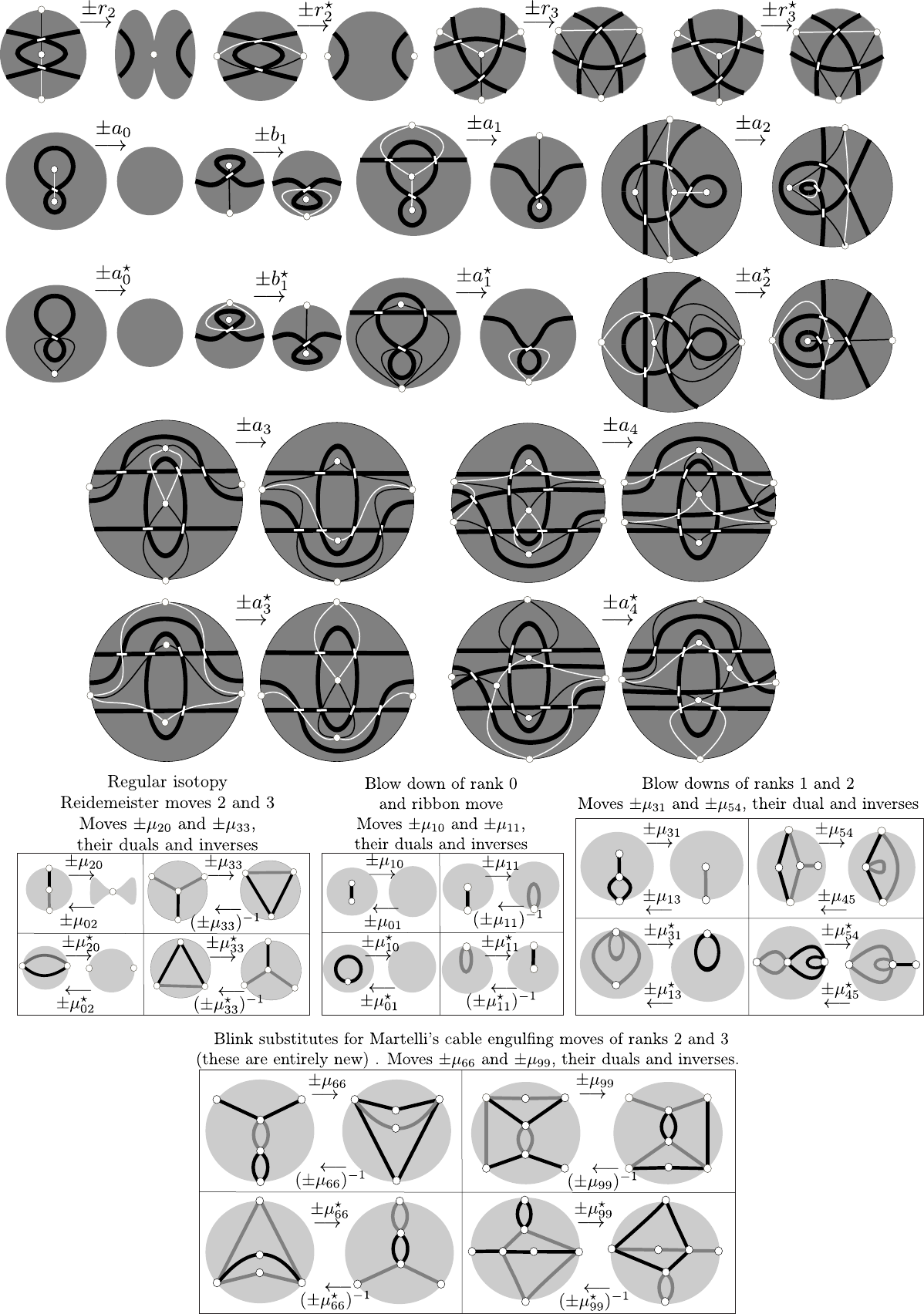}
\caption{\sf 0-Flink and blink  versions of Martelli's revised calculus with 
$a_3$, $a_4$ replacing $A_3$, $A_4$:
In the upper part of the figure, distinctly 2-face colored 0-flinks 
and respective blinks are superimposed. This implies the moves for 
the coin calculus on blinks in the lower part of the figure, 
concluding the proof of the Theorem 1.1.}
\label{fig:flinkandlinktogether}
\end{center}
\end{figure}

\section{The role of the ribbon moves  $\pm \mu_{11}^{\pm 1}$}
\label{sec:roleribbonmoves}
The counterpart of the ribbon moves in the coin calculus, 
(also called ribbon moves) $\pm\mu_{11}^{\pm 1}$ are redundant 
because Martelli's calculus is in $\mathbb{R}^3$. We include them 
in the coin calculus because with their inclusions all the dual moves,
except $\mu_{20}^\star$ and $\mu_{02}^\star$, become redundant.
\begin{lemma}
Let $f$ be the external infinite face of a decorated general position 
link projection and $g$ be a face adjacent to $f$. Then it is possible 
to interchange $f$ and $g$ by means of regular isotopies and
one ribbon move.
\end{lemma}
\begin{proof}
 The proof is given in Fig. \ref{fig:changeexternalface}.
\end{proof}

\begin{figure}[H]
\begin{center}
\includegraphics[width=16cm]{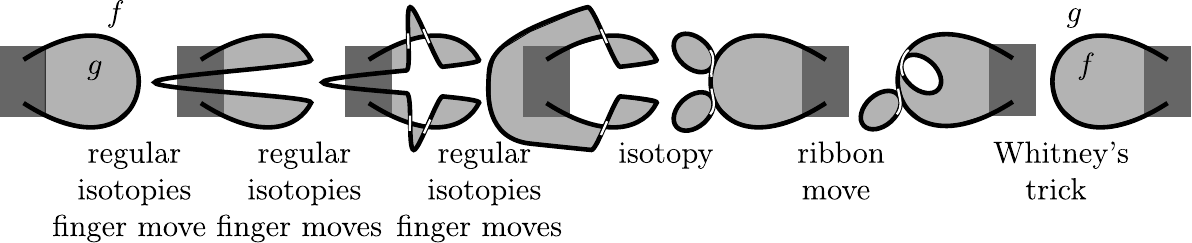}
\caption{\sf Changing the external face $f$ to be any one of its 
adjacent faces $g$ by means of $r_2$, $r_3$ (regular isotopies) 
and the ribbon move: the passage from the first to the 
second configurations is a {\em finger move} where a small segment of an edge is
arbitrarily deformed as a {\em finger} passing over all the crossings of the 
dark gray rectangle; the finger move can be factored by regular isotopies; by other
four finger moves one can go from the second to the third to the fourth 
configuration; the passage from the fourth to the fifth 
configuration is simply an isotopy; from the fifth to the sixth is a ribbon move;
finally, the passage from the sixth to the seventh is acomplished by
Whitney's trick, depicted in the proof below; Whitney's trick also factors as
regular isotopies, as shown in the proof of Corollary \ref{cor:soinsimplification}.
The net effect is to interchange the faces $f$ and $g$: the initial infinite face is $f$ 
and, at the end, the infinite face is $g$.
}
\label{fig:changeexternalface}
\end{center}
\end{figure}

\begin{corollary}
\label{cor:soinsimplification}
 The coin calculus can be simplified to include only the following set of 36 moves
 $ \{\pm \mu_{20}, \pm \mu_{02}, \pm \mu_{20}^\star, \pm \mu_{02}^\star, 
 \pm \mu_{33}^{\pm1}, \pm \mu_{01}, \pm \mu_{01},  \pm \mu_{11}^{\pm -1}, \pm \mu_{31},  \pm \mu_{13},  
 \pm \mu_{54},  \pm \mu_{45}, \pm \mu_{66}^{\pm},\pm \mu_{99}^{\pm}\}.$
\end{corollary}
\begin{proof}
We work in the language of flinks.
Moves corresponding to $(\pm \mu_{11}^{\pm1})^\star$ and  $(\pm \mu_{33}^{\pm1})^\star$ are redundant 
because $\pm \mu_{11}^{\pm1}$  $\pm \mu_{33}^{\pm1}$ are self dual. 
Moves $\pm\mu_{10}^\star$ and  $\pm\mu_{01}^\star$ are implied by a combination of moves 
$\pm \mu_{10}$,  $\pm \mu_{01}$,  $\pm \mu_{11}$ and $(\pm \mu_{11})^{-1}$. 
In particular, all the Reidemeister 2 and 3 moves (regular isotopy) 
are at our disposal. We can use this fact to change the external face of
the link diagram to become any chosen adjacent face by using regular isotopy at the cost of creating two 
curls adjacent in the same component with distinct sign and the same rotation number. 
Use the the appropriate ribbon move to
obtain two curls with the distinct signs and distinct rotation number. Now apply Whitney's trick
which cancel these opposite curls by using $r_2^{\pm 1}$ and $r_3^{\pm 1}$ moves
(regular homotopies):
\raisebox{-3mm}{\includegraphics[width=45mm]{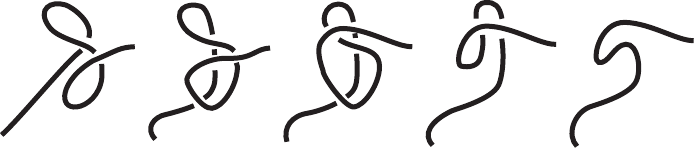}}.
The net effect in the correponding 
final blink  is that it is obtained from the initial blink by dualizing 
and interchanging black and gray
edges. Having this double involutions at our disposal it
is straighforward to obtain all the remaining dual moves.
\end{proof}
In Fig. \ref{fig:reducedblinkcalculus} the 36 moves of the final reduced coin calculus are presented.
\begin{figure}[H]
\begin{center}
\includegraphics[scale=0.93]{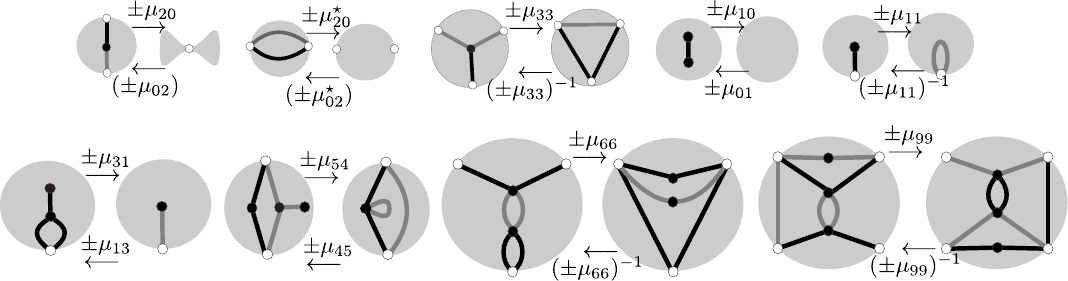}
\caption{\sf The 36 moves (and 36 coins not necessarily distinct) 
defining the reduced coin calculus.}
\label{fig:reducedblinkcalculus}
\end{center}
\end{figure}

\section{Conclusion}
\label{sec:conclusion}
This work is concluded by presenting below a complete census of the 
$k$-small 3-manifolds, for $k=8$. These are the closed, oriented, connected and prime 3-manifolds 
induced by a blink with at most $k$ edges. The first author acknowledges the partial financial support 
of CNPq-Brazil,process number 302353/2014-3. The second author acknowledges the financial support of
FACEPE, IBPG-1295-1.03/12.

\begin{figure}[H]
\begin{center}
\includegraphics[scale=1.15]{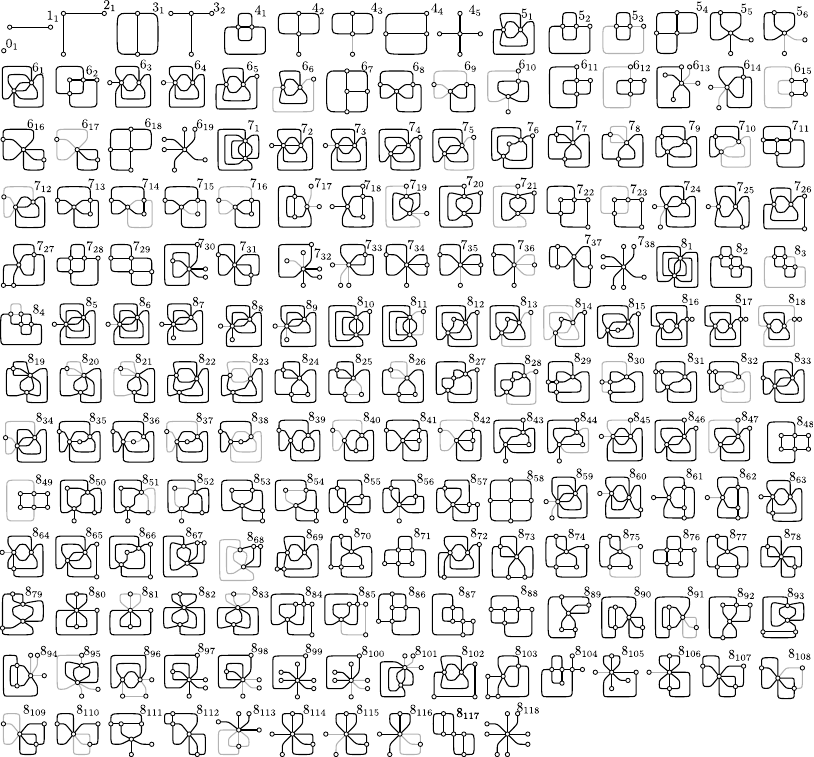}
\caption{\sf A complete census (no misses, no duplicates) of 8-small prime 3-manifolds. 
They correspond to the first (by lexicography) 191 closed, oriented, connected and prime 3-manifolds. 
These are such 3-manifolds which are induced by blinks up to $k=8$ edges. 
The {\em sequence of WRT-invariants} of a closed oriented 
connected 3-manifold is an infinite sequence of complex numbers, indexed by $r\ge3$. 
The $r$-th WRT-invariant of a manifold is directly obtained 
from a blink inducing it. An entirely 
combinatorial recipe directly implementable to compute the WRT-invariants of a 3-manifold from 
a blink inducing it is given in Chapter 7 of \cite{lins1995gca}.
This recipe, in its turn is justified by at the very basic level, also by the combinatorial
theory developed in \cite{kauffman1994tlr}.
Recall that a blink is a finite plane graph with an (arbitrary) edge bipartition. Such 
census are possible by completely combinatorial methods: we generate a subset of 
blinks that misses no 3-manifold by lexicography and the theory in \cite{lins2007blink};
then we compute the homology and the WRT-invariants; at this level $k=8$ 
these two invariants are seen to be complete.  There are 3 independent implementations 
of to obtain the Kauffman-Lins version (depending upon the Temperley-Lieb algebra)
of the WRT-invariants. They were implemented by different people,
in non-overlapping times:
S. Lins (1990-1995), S. Melo (1999-2001) and L. Lins (2006-2007). The results agree.
The software BLINK (\cite{lins2007blink})  
computes the WRT-invariants and the above census by exhaustive
generation of blinks. BLINK also draws links and 0-flinks in a grid,
as a full strength application of network flow theory, using Tamassia's algorithm
\cite {tamassia1987egg}.
BLINK is currently the object of an open source Github project, 
and is available upon request. 
}
\label{fig:prime191Blinks-white-vertices}
\end{center}
\end{figure}

-----------------------------------
\bibliographystyle{spmpsci}
\bibliography{bibliography}

\end{document}